\documentclass{amsart}
\usepackage{amsfonts,amsmath,amsthm,amssymb}

\newtheorem{theorem}{Theorem}
\newtheorem{corollary}{Corollary}
\newtheorem{lemma}{Lemma}

\def\gp#1{\langle #1 \rangle}
\def\m1{^{-1}}

\begin{document}

\title{On the unit group of a commutative group ring}

\author{V. Bovdi and M. Salim}
\dedicatory{Dedicated to  Professor  W.~Kimmerle  on his  60th birthday}

\address{
Department of Math. Sciences,
UAE University - Al-Ain,
\newline
United Arab Emirates
}
\email{vbovdi@gmail.com; MSalim@uaeu.ac.ae}

\thanks{This  paper was supported by PPDNF at UAEU}


\begin{abstract}
We investigate  the group of normalized units of the group algebra $\mathbb{Z}_{p^e}G$
of a finite  abelian $p$-group $G$ over the ring $\mathbb{Z}_{p^e}$  of residues   modulo $p^e$
with $e\geq 1$.
\end{abstract}
\maketitle

\section{Introduction}

Let $V(RG)$ be the group of normalized units of the group ring $RG$  of  a finite abelian $p$-group $G$ over  a commutative ring $R$ of characteristic $p^e$ with   $e\geq 1$.  It is well known (\cite{Bovdi_book}, Theorem 2.10, p.10) that   $V(RG)=1+\omega(RG)$, where
\[
\omega(RG)=\big\{\sum_{g \in G}a_gg \in RG \quad \vert \quad\sum_{g\in G}a_g=0\big\}
\]
is the  augmentation ideal of $RG$.

In the case when $\mathrm{char}(R)=p$ and $G$ is an arbitrary finite (not necessary abelian)  $p$-group, the structure of  $V(RG)$
has been studied by  several authors (see the survey \cite{Bovdi_survey}). For a finite abelian $p$-group  $G$,  the  invariants and the basis of $V(\mathbb{Z}_{p}G)$ has been  given by B.~Sandling (see \cite{Sandling_II}). In general, when $\mathrm{char}(R)=p^e$ with $e\geq 2$,  the structure of the abelian $p$-group $V(RG)$ is still not understood.

In the present paper we  investigate the   invariants of  $V(RG)$ in the case when $R=\mathbb{Z}_{p^e}$ is the  ring of residues  modulo $p^e$.
The question about the bases of $V(\mathbb{Z}_{p^e}G)$ is
left open. Our research can be considered as a natural  continuation of  results  of R.~Sandling.

Note that the investigation  of the group $V(\mathbb{Z}_{p^e}G)$
was started by  F.~Raggi (see for example, \cite{Raggi}). We shall revisit his work done in \cite{Raggi} in order to get a more transparent description of the group $V(\mathbb{Z}_{p^e}G)$.

Several results concerning $RG$ and  $V(RG)$ have found  applications in coding theory, cryptography and threshold logic (see
\cite{Bovdi_threshold, Anokhin, Hurley_II, Hurley_I, Willems}).

\section{Main results}

We start to study $V(\mathbb{Z}_{p^e}G)$ with the description of
its  elements of order $p$.  It is easy to see that if  $z\in
\omega(RG)$ and  $c\in G$ is of order $p$, then    $c +p^{e-1}z$
is a nontrivial unit of order $p$ in $\mathbb{Z}_{p^e}G$. We can
ask whether the converse is true, namely  that every element of
order $p$ in $V(\mathbb{Z}_{p^e}G)$ has   the form $c +p^{e-1}z$,
where $z\in \omega(RG)$ and  $c\in G$ of order $p$. The first
result gives  an affirmative  answer to this question.

\begin{theorem}\label{T:1}
Let $V(\mathbb{Z}_{p^e}G)$ be the group of normalized units of the
group ring $\mathbb{Z}_{p^e}G$  of  a finite abelian $p$-group
$G$, where $e\geq 1$. Then every unit $u\in V(\mathbb{Z}_{p^e}G)$
of order $p$ has a form $u=c+p^{e-1}z$, where $c\in G[p]$ and
$z\in \omega(\mathbb{Z}_{p^e}G)$. Moreover,
\[
V(\mathbb{Z}_{p^e}G)[p]= G[p] \times \big(1+p^{e-1}\omega(\mathbb{Z}_{p^e}G)\big),
\]
where the order of the elementary  $p$-group $1+p^{e-1}\omega(\mathbb{Z}_{p^e}G)$ is $p^{|G|-1}$.
\end{theorem}

A full description of  $V(\mathbb{Z}_{p^e}G)$ is given by the next theorem.

\begin{theorem}\label{T:2}
Let $V(\mathbb{Z}_{p^e}G)$ be the group of normalized units of the group ring $\mathbb{Z}_{p^e}G$  of  a finite
abelian $p$-group $G$ with  $\mathrm{exp}(G)=p^n$ where  $e\geq 1$. Then
\[
\begin{split}
V(\mathbb{Z}_{p^e}G)&=G\times \mathfrak{L}(\mathbb{Z}_{p^e}G),\\
\mathfrak{L}(\mathbb{Z}_{p^e}G)&\cong  lC_{p^{e-1}}\times \Big(\times_{i=1}^{n} s_{i} C_{p^{d+e-1}}\Big),\\
\end{split}
\]
where   the nonnegative integer  $s_i$ is equal to the difference of
\[
|G^{p^{i-1}}|-2|G^{p^{i}}|+|G^{p^{i+1}}|
\]
and the number of cyclic subgroups of order $p^i$ in  the group $G$  and where
$l=|G|-1-(s_1+\cdots+s_n)$.
\end{theorem}

Note that Lemma 9 itself  can be considered as a separate result.

\section{Preliminaries}

If $H$ is a subgroup of $G$, then  we denote the left   transversal of $G$ with respect to $H$  by $\mathfrak{R}_l(G/H)$.
We denote the ideal of $FG$ generated by the elements $h-1$ for  $h\in H$ by ${\mathfrak I}(H)$. Furthermore $FG/{\mathfrak I}(H)\cong F[G/H]$  and
\[
V(FG)/\big(1+ {\mathfrak I}(H)\big)\cong V\big(F[G/H]\big).
\]
Denote  the subgroup of $G$ generated by elements of order $p^n$ by $G[p^n]$.

We  start with the following well known results.
\begin{lemma}\label{L:1}
Let $p$ be a prime and $j=p^lk$, where $(k,p)=1$. If  $l\le  n$, then  $p^{n-l}$ is the largest $p$-power divisor of the binomial coefficient ${p^n \choose j}$.
\end{lemma}

\begin{proof}
If  $j=p^lk$ and  $(k,p)=1$, then the statement follows from
\[
{p^n\choose j}=p^{n-l}\cdot \frac{\prod_{i=1, (i,p)=1}^{j-1} (p^{n}-i)\cdot \prod_{i=1}^{p^{l-1}k}(p^{n}-pi)}{(j-1)!k}.
\]
\end{proof}

\begin{lemma}\label{L:2}
Let $G$ be a finite $p$-group and let $R$ be a commutative ring of characteristic  $p^e$ with   $e\geq 1$.   If  $l\geq e$  then
\[
(1-g)^{p^l}=(1-g^{p^s})^{p^{(l-s)}}, \qquad \quad(s=0,\ldots, l-e+1).
\]
\end{lemma}
\begin{proof}
See Lemma 2.4 in \cite{Coleman_Easdown}.\end{proof}

Let $R$ be a commutative ring of characteristic  $p^e$ with   $e\geq 1$. The ideal
$\omega(RG)$ is nilpotent (\cite{Bovdi_book}, Theorem 2.10, p.10) and the $nth$ power $\omega^n(RG)$  determines  the so-called  {\bf $nth$ dimension subgroup}
\[
\mathfrak{D}_n(RG)=G\cap\big(1+ \omega^n(RG)\big), \qquad (n\geq 1).
\]
\begin{lemma}\label{L:3}
(see 1.14, \cite{Sandling_I})
Let $e\geq 1$. If  $G$ is  a finite abelian $p$-group, then
\[
\mathfrak{D}_n(\mathbb{Z}_{p^e}G)=\begin{cases}
G, &  \text{ if }\quad  n=1;\\
G^{p^{e+i}}, & \text{if}\quad  \quad  p^i<n\leq p^{i+1}.
\end{cases}
\]
\end{lemma}

The next two lemmas are  well known.
\begin{lemma}\label{L:4}
If  $G$ is   a finite  abelian $p$-group, then
\[
V(\mathbb{Z}_{p}G)[p]=1+\mathfrak I(G[p]).
\]
\end{lemma}
\begin{proof}
See Lemma 3.3 in \cite{Bovdi_Grishkov}. \end{proof}

\begin{lemma}\label{L:5}
Let $U(R)$ be the group of units of a commutative  ring $R$ with $1$.
If  $I$ is  a nilpotent ideal  in $R$, then
\[
U(R)/ (1+I) \cong U(R/I)
\]
and the group $(1+ I^m) /( 1+ I^{m+1})$ is isomorphic to the additive
group of the quotient  $I^m / I^{m+1}$.
\end{lemma}
\begin{proof} Note that $I$ is the kernel of the natural epimorphism $\sigma:
R\to R/I$. On  $U(R)$ the map $\sigma$ induces
the group homomorphism  $\tilde{\sigma}: U(R)\to U(R/I)$
which is an epimorphism with kernel $1+I$.  Indeed, if $x+I\in
U(R/I)$ and $\sigma(w)=(x+I)^{-1}$, then
\[
\sigma(xw)=(x+I)(x+I)^{-1}=1+I.
\]
Thus $xw= 1+t$ for some $t\in I$
and $1+t$ is a unit in $R$, so $w\in U(R)$. Of course
$x=w^{-1}(1+t)$ is a unit such that
$\tilde{\sigma}(w)=(x+I)^{-1}=x^{-1}+I$. Therefore
$\tilde{\sigma}:U(R)\to U(R/I)$ is an epimorphism.

Now, let $x, y\in I^{m}$ and put $\psi(1+x)= x+ I^{m+1}$.
Then
\[
\begin{split}
\psi((1+x)(1+y))&= xy+x+y+I^{m+1} \\
&= x+y+ I^{m+1}= \psi(1+x)
+\psi(1+y),
\end{split}
\]
so  $\psi$ is a homomorphism of the multiplicative group $1+ I^m$
to the additive group $I^{m} / I^{m+1}$ with kernel $1+I^{m+1}$.
\end{proof}

\bigskip

Let $f_e: \mathbb{Z}_{p^e}\to \mathbb{Z}_{p^{e-1}}$ $(e\geq 2)$
be a ring homomorphism  determined   by
\[
f_e(a+(p^e))=a+(p^{e-1})\qquad \quad (a\in \mathbb{Z}).
\]
Clearly $\mathbb{Z}_{p^e}/(p^{e-1} \mathbb{Z}_{p^e})\cong
\mathbb{Z}_{p^{e-1}}$ and the homomorphism $f_{e}$ can be linearly extended  to the
group ring homomorphism
\begin{equation}\label{E:1}
\overline{f_e}:\mathbb{Z}_{p^e}G \rightarrow \mathbb{Z}_{p^{e-1}}G.
\end{equation}
Let us define the map  $\frak{r}: \mathbb{Z}_{p^e}\to \mathbb{Z}$ to be  the map with the property that for any integer $\alpha$ with
$0\leq \alpha< p^e-1$ we have $\frak{r}\m1(\alpha)=\overline{\alpha}\in \mathbb{Z}_{p^e}$.
Obviously,    $\mathbb{Z}_{p^e}\ni \gamma_g=\alpha_g+p^{e-1}\beta_g$, where $0\leq \frak{r}(\alpha_g) <p^{e-1}$. Hence any $x\in \mathbb{Z}_{p^e}G$ can be written as
\begin{equation}\label{E:2}
x=\sum_{g\in G}\gamma_g g=\sum_{g\in G}\alpha_gg+p^{e-1}\sum_{g\in G}\beta_gg,
\end{equation}
where  $red_p(x)=\sum_{g\in G}\alpha_gg\in \mathbb {Z}_{p^e}G$  is called the {\bf
$p$-reduced part} of $x$.

It is easy to see, that  $\mathfrak{Ker}(\overline{f_e})=p^{e-1}
\mathbb{Z}_{p^e}G$ and
$\big(\mathfrak{Ker}(\overline{f_e})\big)^2=0$, so by (\ref{E:1}) and (\ref{E:2}) we obtain that
\[
\mathbb{Z}_{p^e}G\Big/(p^{e-1}\mathbb{Z}_{p^e}G)\cong  \mathbb{Z}_{p^{e-1}}G.
\]
Since  $p^{e-1}\mathbb{Z}_{p^e}G$ is a nilpotent ideal by Lemma \ref{L:5},
\[
U(\mathbb{Z}_{p^e}G)/(1+p^{e-1}\mathbb{Z}_{p^e}G)\cong U(\mathbb{Z}_{p^{e-1}}G).
\]
Clearly, $1 +p^{e-1}\mathbb{Z}_{p^e}G$ is an elementary abelian $p$-group of order
\[
\begin{split}
|1+p^{e-1}\mathbb{Z}_{p^e}G|=\textstyle\frac{|V(\mathbb{Z}_{p^{e}}G)|\cdot |U(\mathbb{Z}_{p^{e}})|}{|V(\mathbb{Z}_{p^{e-1}}G)|\cdot |U(\mathbb{Z}_{p^{e-1}})|}&\\
=\frac{ p^{e(|G|-1)}\cdot p} {p^{(e-1)(|G|-1)}}&=p^{|G|}.
\end{split}
\]
Furthermore, if  $u\in V(\mathbb{Z}_{p^e}G)=1+\omega(\mathbb{Z}_{p^e}G)$, then
\[
u=red_p(u)+p^{e-1}\sum_{g\in G}\beta_g(g-1),
\]
where  $red_p(u)=1+\sum_{g\in G}\alpha_g(g-1)$ is a unit and $0\leq \frak{r}(\alpha_g) < p^{e-1}$. It follows that
\begin{equation}\label{E:3}
u=red_p(u)(1+p^{e-1} z), \qquad \quad(z\in \omega(\mathbb{Z}_{p^e}G)).
\end{equation}

\begin{lemma}\label{L:6}
Let $\overline{f_e}: V(\mathbb{Z}_{p^e}G) \rightarrow V(\mathbb{Z}_{p^{e-1}}G)$ be the group homomorphism naturally obtained  from (\ref{E:1}).
Then $\frak{Ker}(\overline{f_e})=1+p^{e-1}\omega(\mathbb{Z}_{p^e}G)$ is an elementary abelian $p$-group of  order $p^{|G|-1}$ and
\begin{equation}\label{E:4}
V(\mathbb{Z}_{p^e}G)/\big(1+p^{e-1}\omega(\mathbb{Z}_{p^e}G)\big)\cong
V(\mathbb{Z}_{p^{e-1}}G).
\end{equation}
\end{lemma}

\begin{proof} Let $u\in V(\mathbb{Z}_{p^e}G)$. Then by (\ref{E:3}) we have that
\[
\overline{f_e}(u)=1+\sum_{g\in G}\big(\alpha_g+(p^{e-1})\big)(g-1)\in V(\mathbb{Z}_{p^{e-1}}G),
\]
so
$V(\mathbb{Z}_{p^e}G)/\big(1+p^{e-1} W\big)\cong
V(\mathbb{Z}_{p^{e-1}}G)$,
where $W\subseteq \omega(\mathbb{Z}_{p^e}G)$. It is easy to check that $1+p^{e-1} W$ is an elementary abelian $p$-group  of order
\[
\begin{split}
|1+p^{e-1}W|&=\textstyle\frac{|V(\mathbb{Z}_{p^{e}}G)|}{|V(\mathbb{Z}_{p^{e-1}}G)|}\\
&=\frac{ p^{e(|G|-1)}} {p^{(e-1)(|G|-1)}}=p^{|G|-1}.
\end{split}
\]
Clearly, \quad   $|p^{e-1}\omega(\mathbb{Z}_{p^e}G)|=|p^{e-1} W|=p^{|G|-1}$   and consequently
\[
1 +p^{e-1}W=1+p^{e-1}\omega(\mathbb{Z}_{p^e}G).
\]
The proof is complete.
\end{proof}

\section{Proof of the Theorems}

\begin{proof}[Proof of  Theorem 1] Use induction on $e$. The base of the induction is:   $e=2$.

Put $H=G[p]$. Any $u\in V(\mathbb{Z}_{p^e}G)[p]$ can be  written as
\[
u=c_1x_1+\cdots +c_tx_t,
\]
where $c_1,\ldots,c_t\in \mathfrak{R}_l(G/H)$ and $x_1,\ldots, x_t\in \mathbb{Z}_{p^2}H$.

First, assume that  $c_i\not\in H$ for any  $i=1,\ldots t$. Clearly,
\begin{equation}\label{E:5}
\overline{f_2}(u)= c_1\overline{f_2}(x_1)+c_2\overline{f_2}(x_2)+\cdots +c_t\overline{f_2}(x_t)\in V(\mathbb{Z}_pG).
\end{equation}
Since  $\overline{f_2}(u)\in 1+\frak{I}(H)$ (see Lemma \ref{L:4}), we have that  $c_j\in H$ for some $j$,  by (\ref{E:5}), a contradiction.

Consequently, we can assume that $c_1=1\in H$,   $x_1\not=0$ and $1\in Supp(x_1)$.
This yields that
\[
\begin{split}
\overline{f_2}(u)= &\overline{f_2}(x_1-\chi(x_1))+c_2\overline{f_2}(x_2-\chi(x_2))+\cdots +c_t\overline{f_2}(x_t-\chi(x_t))\\
+ &\overline{f_2}(\chi(x_1))+c_2\overline{f_2}(\chi(x_2))+\cdots +c_t\overline{f_2}(\chi(x_t))\in V(\mathbb{Z}_pG).
\end{split}
\]
Clearly,   either $\overline{f_2}(u)=1$ or $\overline{f_2}(u)$ has  order $p$.  Lemma \ref{L:4} ensures that
\[
\begin{split}
f_2(\chi(x_1))&\equiv 1\pmod{p},\qquad \text{and}\\
f_2(\chi(x_2))\equiv \cdots \equiv f_2(\chi(x_t))&\equiv 0\pmod{p}.\\
\end{split}
\]
It follows that  $u$  can be written  as
\begin{equation}\label{E:6}
u=1+\sum_{i=1}^tc_i\sum_{h\in G[p]}\beta_{h}^{(i)}(h-1)+pz, \qquad \quad(z\in \mathbb{Z}_{p^2}G).
\end{equation}
We can assume that $z=0$. By Lemma \ref{L:3}  we have that
\[
G=\mathfrak{D}_1(G)\supset \mathfrak{D}_2(G)=\mathfrak{D}_3(G)=\cdots =\mathfrak{D}_p(G)=G^{p^2},
\]
so (\ref{E:6}) can be rewritten  as
\[
u=1+\sum_{i=1}^tc_i\sum_{h\in G[p]\setminus \mathfrak{D}_p}\beta_{h}^{(i)}(h-1) +w,
\]
where $w\in \omega^2(\mathbb{Z}_{p^2}G)$. Then, by the binomial  formula,
\[
\begin{split}
1&=u^p\equiv 1+ p\sum_{i=1}^tc_i\sum_{h\in G[p]\setminus \mathfrak{D}_p}\beta_{h}^{(i)}(h-1)\\
&+\binom{p}{2}\Big(\sum_{i=1}^tc_i\sum_{h\in G[p]\setminus \mathfrak{D}_p}\beta_{h}^{(i)}(h-1)\Big)^2+
\cdots \quad \pmod{\omega^2(\mathbb{Z}_{p^2}G)}.
\end{split}
\]
It follows that
\[
p\sum_{i=1}^tc_i\sum_{h\in G[p]\setminus \mathfrak{D}_p}\beta_{h}^{(i)}(h-1)\equiv 0 \pmod{\omega^2(\mathbb{Z}_{p^2}G)}
\]
and\quad $\beta_{h}^{(i)}\equiv 0 \pmod{p^2}$\quad   for any \quad $h\in G[p]\setminus \mathfrak{D}_p$.
Hence by (\ref{E:6}),
\[
u=1+w, \qquad(w\in \omega^2(\mathbb{Z}_{p^2}G)).
\]
Again, by Lemma \ref{L:3},  we have that
\[
G^{p^2}= \mathfrak{D}_p(G)\supset \mathfrak{D}_{p+1}(G)=\mathfrak{D}_{p+2}(G)=\cdots =\mathfrak{D}_{p^2}(G)=G^{p^3}
\]
and (\ref{E:6}) can be rewritten  as
\[
u=1+\sum_{i=1}^tc_i\sum_{h\in \mathfrak{D}_{p}\setminus \mathfrak{D}_{p^2}}\beta_{h}^{(i)}(h-1) +w,
\]
where $w\in \omega^3(\mathbb{Z}_{p^2}G)$. This yields
\[
\begin{split}
1=u^p\equiv 1&+ p\sum_{i=1}^tc_i\sum_{h\in \mathfrak{D}_{p}\setminus \mathfrak{D}_{p^2}}\beta_{h}^{(i)}(h-1)\\  &+\binom{p}{2}\Big(\sum_{i=1}^tc_i\sum_{h\in \mathfrak{D}_{p}\setminus \mathfrak{D}_{p^2}}\beta_{h}^{(i)}(h-1)\Big)^2+
\cdots \pmod{\omega^3(\mathbb{Z}_{p^2}G)}.
\end{split}
\]
As before, it follows that
\[
p\sum_{i=1}^tc_i\sum_{h\in \mathfrak{D}_{p}\setminus \mathfrak{D}_{p^2}}\beta_{h}^{(i)}(h-1)\equiv 0 \pmod{\omega^3(\mathbb{Z}_{p^2}G)}
\]
and\quad $\beta_{h}^{(i)}\equiv 0 \pmod{p^2}$\quad for any\quad $h\in \mathfrak{D}_{p}\setminus \mathfrak{D}_{p^2}$. Therefore,
\[
u=1+w,\qquad\qquad   (w\in \omega^3(\mathbb{Z}_{p^2}G)).
\]
By continuing  this process  we obtain that $u=1+pv$, because the augmentation  ideal $\omega(\mathbb{Z}_{p^2}G)$ is nilpotent.

Now assume that  the statement of our lemma is true for $\mathbb{Z}_{p^{e-1}}G$.
This means  that for a unit  $u$ of the form (\ref{E:3}) we get
$\beta_{h}^{(i)}=p^{e-2}\alpha_{h}^{(i)}$ and
\[
u=1+\sum_{i=1}^tc_i\sum_{h\in G[p]}p^{e-2}\alpha_{h}^{(i)}(h-1),
\]
so
\[
1=u^p=1+p\sum_{i=1}^tc_i\sum_{h\in G[p]}p^{e-2}\alpha_{h}^{(i)}(h-1)
\]
and \quad $\alpha_{h}^{(i)}\equiv \;0\pmod{p}$. The proof is complete.
\end{proof}
\begin{lemma}\label{L:7}
Let $G$ be a finite abelian $p$-group. Then
\[
V(\mathbb{Z}_{p^e}G)=G\times \mathfrak{L}(\mathbb{Z}_{p^e}G)
\]
and  the following conditions hold:
\begin{itemize}
\item[(i)]  if $e\geq 2$, then  $\overline{f_e}\big(\mathfrak{L}(\mathbb{Z}_{p^e}G)\big)= \mathfrak{L}(\mathbb{Z}_{p^{e-1}}G)$;
\item[(ii)] if $e\geq 2$, then $\mathfrak{Ker}(\overline{f_e})=1+p^{e-1}\omega(\mathbb{Z}_{p^e}G)=\mathfrak{L}(\mathbb{Z}_{p^e}G)[p]$  and
\begin{equation}\label{E:7}
\mathfrak{L}(\mathbb{Z}_{p^e}G) /\big(1+p^{e-1}\omega(\mathbb{Z}_{p^e}G)\big)\cong \mathfrak{L}(\mathbb{Z}_{p^{e-1}}G);
\end{equation}
\item[(iii)] $\mathfrak{L}(\mathbb{Z}_{p^e}G)[p]\cong \mathfrak{L}(\mathbb{Z}_{p^{e-1}}G)[p]$ for $e\geq 3$.
\end{itemize}
\end{lemma}
\begin{proof}
If $e=1$, then there exists a subgroup $\mathfrak{L}(\mathbb{Z}_pG)$ of $V(\mathbb{Z}_pG)$ (see \cite{Johnson}, Theorem 3) such that
$V(\mathbb{Z}_pG)=G\times \mathfrak{L}(\mathbb{Z}_pG)$.

Assume
$V(\mathbb{Z}_{p^{e-1}}G)=G\times \mathfrak{L}(\mathbb{Z}_{p^{e-1}}G)$.
Consider the homomorphism
\[
\overline{f_e} : V(\mathbb{Z}_{p^e}G)\rightarrow V(\mathbb{Z}_{p^{e-1}}G)=G\times \mathfrak{L}(\mathbb{Z}_{p^{e-1}}G).
\]
Denote  the preimage of $\mathfrak{L}(\mathbb{Z}_{p^{e-1}}G)$ in  $V(\mathbb{Z}_{p^e}G)$ by  $\mathfrak{L}(\mathbb{Z}_{p^e}G)$.
Clearly,  $\overline{f_e}(g)=g$ for all $g\in G$  and
\[
\frak{Ker}(\overline{f_e})=1+{p^{e-1}}\omega(\mathbb{Z}_{p^e}G)\leq \mathfrak{L}(\mathbb{Z}_{p^e}G).
\]
If $x \in \mathfrak{L}(\mathbb{Z}_{p^e}G)\cap G$, then
\[
G\ni \overline{f_e}(x)\in  \mathfrak{L}(\mathbb{Z}_{p^{e-1}}G)\cap G=\gp{1},
\]
 so $x=1$.
Hence $\mathfrak{L}(\mathbb{Z}_{p^e}G)\cap G=\gp{1}$ and $G\times \mathfrak{L}(\mathbb{Z}_{p^e}G)\subseteq V(\mathbb{Z}_{p^e}G)$. Since
\[
\overline{f_e}(G\times \mathfrak{L}(\mathbb{Z}_{p^e}G))= V(\mathbb{Z}_{p^{e-1}}G)
\]
and $\frak{Ker}(\overline{f_e})\subseteq G\times \mathfrak{L}(\mathbb{Z}_{p^e}G)$, we have that
$V(\mathbb{Z}_{p^e}G)=G\times \mathfrak{L}(\mathbb{Z}_{p^e}G)$ by properties of the homomorphism.

(ii) Clearly the epimorphism  $\overline{f_e}$  $(e\geq 2)$ satisfies  (\ref{E:7}) by construction.

(iii) Let $e\geq 3$. From (ii) we have
\[
\mathfrak{Ker}(\overline{f_e})=1+p^{e-1}\omega(\mathbb{Z}_{p^e}G))=\mathfrak{L}(\mathbb{Z}_{p^e}G)[p]
\]
and $|1+p^{e-1}\omega(\mathbb{Z}_{p^e}G))|=p^{|G|-1}$ (see Lemma \ref{L:6}). It follows that
\[
|\mathfrak{L}(\mathbb{Z}_{p^e}G)[p]|=|\mathfrak{L}(\mathbb{Z}_{p^{e-1}}G)[p]|=p^{|G|-1},
\]
so the proof is finished.
\end{proof}

\begin{lemma}\label{L:8}
Let $e\geq 2$.
If $u\in \mathfrak{L}(\mathbb{Z}_{p^e}G)$, then
\begin{equation}\label{EE:8}
|u|=p\cdot |\overline{f_e}(u)|.
\end{equation}
\end{lemma}
\begin{proof}
Let $|u|=p^{m}$. By Theorem  \ref{T:1} we obtain that  $u^{p^{m-1}}=1+p^{e-1}z$  for some $z\in \omega(\mathbb{Z}_{p^{e}}G)$, and
$\overline{f_e}(u^{p^{m-1}})=1$, so the statement follows by induction.
\end{proof}

\begin{lemma}\label{L:9}
Let  $d\geq 1$ and  $0\not=y\in \mathbb{Z}_{p^{e}}G$. Then $(1+p^dy)^{p^{e-d}}=1$ and the following conditions hold:
\begin{itemize}
\item[(i)] if  $p^{e-1}y\not=0$,  then  the unit $1+p^d y$ has order $p^{e-d}$, except  when
\[
p=2,\quad  d=1\quad  \text{and}\quad  y^2\not\in  2\mathbb{Z}_{2^e}G;
\]

\item[(ii)] if $p^{e-1}y=0$ then  $y=p^sz$, where $p^{e-1}z\not=0$, and  the unit $1+p^{d+s}z$ has order
$p^{e-d-s}$.
\end{itemize}
\end{lemma}

\begin{proof} Let $j=p^lk$ and  $(k,p)=1$. By Lemma \ref{L:1}, the number
$p^{e+(j-1)d-l}$ is the largest $p$-power divisor of  $\binom{p^{e-d}}{j} p^{jd}$\quad  for $j\geq 1$. Since
\[
\begin{split}
e-d-l+p^l kd&\geq e-d-l +p^ld\\
&=e+(p^l-1)d-l\geq e+p^l-1-l\geq e;\\
dp^{e-d}&\geq d+p^{e-d}\\
&\geq d+e-d\geq e,\\
\end{split}
\]
the number $p^e$ divides the natural  numbers  $\binom{p^{e-d}}{j}p^{jd}$ and  $p^{dp^{e-d}}$.
Using these inequalities, we have
\[
\textstyle
(1+p^{d}y)^{p^{e-d}}=1 +\sum_{j=1}^{p^{e-d}}\binom{p^{e-d}}{j}p^{jd}\cdot y^j+ p^{dp^{e-d}}\cdot  y^{p^{e-d}}=1.
\]
Therefore, the order of $1+p^{d}y$ is a divisor of  $p^{e-d}$.

Assume that $(1+p^{d}y)^{p^{e-d-1}}=1$. Since
\[
\begin{split}
dp^{e-d-1}&\geq d+p^{e-d-1}\\
&\geq d+1+(e-d-1)\geq e,\\
\end{split}
\]
we obtain that
\[
\begin{split}
\textstyle
(1+p^{d}y)^{p^{e-d-1}}&=1+ \sum_{j=1}^{p^{e-d-1}-1}\binom{p^{e-d-1}}{j}p^{jd}y^j+p^{dp^{e-d-1}}\cdot y^{p^{e-d-1}}\\
&=1+ \sum_{j=1}^{p^{e-d-1}-1}\binom{p^{e-d-1}}{j}p^{jd}y^j\\
&=1
\end{split}
\]
and \quad   $\sum_{j=1}^{p^{e-d-1}-1}\binom{p^{e-d-1}}{j}p^{jd}y^j=0$.\quad  This yields that
\begin{equation}\label{EE:9}
p^{e-1}y=-\binom{p^{e-d-1}}{2}p^{2d}y^2- \sum_{j=3}^{p^{e-d-1}-1}\binom{p^{e-d-1}}{j}p^{jd}y^j.
\end{equation}
Assume that $p^{e-1}y\not=0$. Since $j=p^lk$,   where $(k,p)=1$, the number
$p^{e+(j-1)d-1-l}$ is the largest $p$-power divisor of  $\binom{p^{e-d-1}}{j} p^{jd}$ for $j\geq 2$   by Lemma \ref{L:1}.
Put $m=(j-1)d-1-l$.

Consider the following cases:

\underline{Case 1.} Let $l=0$. Then  $m=(k-1)d-1-l$ and   $k\geq 2$, so  $m\geq 0$.

\underline{Case 2.} Let $l>1$. Then $j=p^lk\geq p^l\geq 4$ and
\[
\begin{split}
m&=(p^lk-1)d-1-l\\
&\geq (p^l-1)-1-l= p^l-2-l\\
&\geq(p^l+l)-l-2=p^l-2\geq 0.
\end{split}
\]

\underline{Case 3.} Let $l=1$. Then $pk>2$ unless  $p=2$ and $d=1$. If $p=2$ and $d=1$ we have
$m=(pk-1)2-2=2pk-4\geq 0$.

In all cases $m\geq 0$ unless  $p=2$, $d=1$ and  $y^2\not\in  2\mathbb{Z}_{2^e}G$.
Therefore
$p^{e+(j-1)d-1-l}\geq p^e$ and   by (\ref{EE:9}), we get  $p^{e-1}y= 0$, a contradiction.
Hence, the order of the unit $1+p^{d}y$ is $p^{e-d}$. The proof of  part (i) is finished.

If $p^{e-1}y=0$ then  $y=p^sz$, where $p^{e-1}z\not=0$, so by part (i),  the unit $1+p^{d+s}z$ has order
$p^{e-d-s}$.
\end{proof}
\begin{corollary}\label{C:1}
If $G=\gp{a\mid a^2=1}$  then
\[
V(\mathbb{Z}_{2^e}G)=G\times \gp{1+2(a-1)}\cong C_2\times C_{2^{e-1}}.
\]
\end{corollary}
\begin{proof}
 Indeed,  $(a-1)^2=-2(a-1)$, so     $|1+2(a-1)|=2^{e-1}$.
\end{proof}

\smallskip

\begin{proof}[Proof of  Theorem 2]
Let  $|V(\mathbb{Z}_p G)[p]| = p^{r}$ and   $\mathrm{exp}(G)=p^n$. Assume that
\begin{equation}\label{E:9}
V(\mathbb{Z}_{p}G)=\gp{b_1}\times\cdots\times\gp{b_r},
\end{equation}
where $|\gp{b_j}|=p^{c_j}$. The  number $r=\mathrm{rank}_p(V)$ is called the {\bf $p$-rank of}  $V(\mathbb{Z}_{p} G)$.
Obviously
\[
V(\mathbb{Z}_{p} G)[p]=\gp{b_1 ^{p^{c_1 -1}}}\times  \gp{b_2 ^{p^{c_2-1}} }\times \cdots\times \gp{b_r ^{p^{c_r -1}}}.
\]
Put $H=G[p]$. Since  $V(\mathbb{Z}_{p}G)[p]=1+\mathfrak I(H)$   (see Lemma \ref{L:4}),   $p^{r}$
equals  the number of  the elements of the ideal  ${\mathfrak I}(H)$. It is well known
(see \cite{Bovdi_book},  Lemma 2.2, p.7) that a basis of  ${\mathfrak I}(H)$ consists of
\[
\{u_i(h_j-1)\mid u_i\in \mathfrak{R}_l(G/H),\quad h_j\in H\setminus 1\}
\]
and the  number of such elements is
${|G|\over|H|}(|H|-1)=|G|-|G^p|$. Hence
\[
r=\mathrm{rank}_p(V)=|G|-|G^p|.
\]
Since    $V(\mathbb{Z}_{p} G)^p = V(\mathbb{Z}_{p} G^{p})$, we have
$\mathrm{rank}_p(V(\mathbb{Z}_{p} G)^p)=|G^p|-|G^{p^2}|$. It follows
that the number of cyclic subgroups of order $p$ in $V(\mathbb{Z}_{p} G)$ (see (\ref{E:9})) is
\[
(|G|-|G^p|)-(|G^p|-|G^{p^2}|)=|G|-2|G^p|+|G^{p^2}|.
\]
Repeating  this argument, one can easily  see that the number of  elements of order $p^i$ in $V(\mathbb{Z}_{p} G)$ is equal to
\begin{equation}\label{EE:11}
|G^{p^{i-1}}|-2|G^{p^{i}}|+|G^{p^{i+1}}|, \qquad\qquad (i=1,\ldots, n).
\end{equation}
Recall that   $V(\mathbb{Z}_{p}G)=G\times
\mathfrak{L}(\mathbb{Z}_{p}G)$ (see \cite{Johnson}, Theorem 3) is
a finite abelian $p$-group and $\mathfrak{L}(\mathbb{Z}_{p}G)$ has
a decomposition
\begin{equation}\label{EE:12}
\mathfrak{L}(\mathbb{Z}_{p}G)\cong \times_{d=1}^{n} s_{d} C_{p^d}\qquad \quad (s_d\in \mathbb{N}),
\end{equation}
where $\mathrm{rank}_p(\mathfrak{L}(\mathbb{Z}_{p}G))=r=s_1+\cdots+s_n$ and  $\mathrm{exp}(G)=p^n$. The number $s_i$ is equal to the difference of  (\ref{EE:11}) and the number of cyclic subgroups of order $p^i$ in the direct decomposition of the group $G$.

\smallskip

We use induction on $e\geq 2$ to prove that
\begin{equation}\label{EE:13}
\mathfrak{L}(\mathbb{Z}_{p^e}G)\cong  lC_{p^{e-1}}\times\Big(\times_{d=1}^{n} s_{d} C_{p^{d+e-1}}\Big),
\end{equation}
where $l=|G|-1-r$   and where $s_1, \ldots, s_n\in \mathbb{N}$ are from (\ref{EE:12}).

The base of the induction is:   $e=2$.  According to Lemma
\ref{L:7}, the kernel  of the epimorphism $\overline{f_e}$ is
$\frak{Ker}(\overline{f_e})=1+p\omega(\mathbb{Z}_{p^2}G)$, which
consists of all elements of order $p$ in
$\mathfrak{L}(\mathbb{Z}_{p^2}G)$ and
$|1+p\omega(\mathbb{Z}_{p^2}G)|=p^{|G|-1}$
by Lemma \ref{L:6}. Hence
\[
\mathrm{exp}(\mathfrak{L}(\mathbb{Z}_{p^2}G))=p\cdot \mathrm{exp}(\mathfrak{L}(\mathbb{Z}_{p}G)=p^{n+1}
\]
and the finite abelian $p$-group  $\mathfrak{L}(\mathbb{Z}_{p^2}G)$ has a decomposition
\[
\mathfrak{L}(\mathbb{Z}_{p^2}G)\cong lC_p  \times\Big(\times_{d=1}^{n} s_{d} C_{p^{d+1}}\Big),
\]
where $s_1, \ldots, s_n\in \mathbb{N}$  are from (\ref{EE:12}),  and where
$l=|G|-1-r$ by Lemma \ref{L:6}.

\smallskip
Assume that
\[
\mathfrak{L}(\mathbb{Z}_{p^{e-1}}G)\cong  lC_{p^{e-2}}\times\Big(\times_{d=1}^{n} s_{d} C_{p^{d+e-2}}\Big).
\]
Using   Lemma \ref{L:8}, we get
\[\mathrm{exp}(\mathfrak{L}(\mathbb{Z}_{p^e}G))=p\cdot \mathrm{exp}(\mathfrak{L}(\mathbb{Z}_{p^{e-1}}G))=p^{n+e-1}\]
and $\mathfrak{L}(\mathbb{Z}_{p^e}G)[p]\cong                                                                                                                                         \mathfrak{L}(\mathbb{Z}_{p^{e-1}}G)[p]$ with $e>2$, by Lemma \ref{L:7}(iii).
Now, again as before, we obtain  (\ref{EE:13}). The proof is complete.
\end{proof}

\end{document}